\documentclass[12pt, reqno, twoside, letterpaper]{amsart}

\usepackage{url}
\usepackage[colorlinks,linkcolor=blue,anchorcolor=blue,citecolor=blue,backref=page]{hyperref}
\usepackage{paperstyle}
\usepackage{mathtools}
\usepackage{graphics,epsfig}
\usepackage{graphicx}
\usepackage{float}
\usepackage{bm}
\usepackage{cite}

\usepackage{todonotes}

\usepackage[norefs,nocites]{refcheck}

\def\e{{\mathbf{\,e}}}
\def\ep{{\mathbf{\,e}}_p}

\def\({\left(}
\def\){\right)}
\def\fl#1{\left\lfloor#1\right\rfloor}
\def\rf#1{\left\lceil#1\right\rceil}
\def\mand{\qquad \mbox{and} \qquad}
\def\vec#1{{\mathbf #1}}

\def \bfalpha{\bm{\alpha}}
\def \bfbeta{\bm{\beta}}

\def \bfeta{\bm{\eta}}
\def \bfzeta{\bm{\zeta}}
\def \bfeta{\bm{\eta}}

\def\sssum{\mathop{\sum\sum\sum}}

\numberwithin{figure}{section}

\def\SAMN{\sS_r(\bfalpha; \cM, \cN)}
\def\SABMN{\sS_r(\bfalpha, \bfbeta ; \cM, \cN)}

\title[Congruences with intervals and arbitrary sets]
      {Congruences with intervals and arbitrary sets}
      
\author[W.\ D.\ Banks]{William Banks}

\address{Department of Mathematics, University of Missouri, Columbia, MO 65211 USA}

\email{bankswd@missouri.edu}
        
\author[I.\ E.\ Shparlinski]{Igor Shparlinski}

\address{Department of Pure Mathematics,
University of New South Wales, Sydney, NSW 2052 Australia}
		 
\email{igor.shparlinski@unsw.edu.au}
           
\date{\today}

\begin{document}

\begin{abstract}
Given a prime $p$, an integer $H\in[1,p)$,
and an arbitrary set $\cM\subseteq \FF_p^*$, 
where $\FF_p$ is the finite field with $p$ elements, let $J(H,\cM)$
denote the number of solutions to the congruence
$$
xm\equiv yn\bmod p
$$
for which $x,y\in[1,H]$ and $m,n\in\cM$. In this paper, we bound
$J(H,\cM)$ in terms of $p$, $H$ and the cardinality of $\cM$. In a wide range of 
parameters, this bound is optimal. We give two applications of this bound: to 
new estimates of trilinear character sums and to bilinear sums with Kloosterman sums,
complementing some recent results of  Kowalski,  Michel and  Sawin~(2018). 
\end{abstract}

\subjclass[2010]{11A07,11L05, 11L40}
\keywords{congruences, character sums, Kloosterman sums}

\maketitle


\section{Introduction}
\subsection{Set up}

For each prime $p$ we denote by $\FF_p$ the finite field with
$p$~elements, which can be identified with the set
\begin{equation}
\label{eq:Fp repr}
\FF_p = \{0,\ldots,p-1\}.
\end{equation}
Given an integer $H\in[1,p)$ and an arbitrary set
$\cM\subseteq \FF_p^*$, let $J(H,\cM)$ be the number
of solutions $(x,y,m,n)$ to the congruence
\begin{equation}
\label{eq:Cong}
xm\equiv yn\bmod p
\end{equation}
for which $x,y\in[1,H]$ and $m,n\in\cM$.
In this paper, we study the problem of bounding $J(H,\cM)$ in terms
of $p$, $H$ and the cardinality $M=|\cM|$.

In what follows, we freely alternate between the language
of congruences and that of equations over finite fields.
For example, in view of the isomorphism $\ZZ/p\ZZ\cong\FF_p$
the congruence~\eqref{eq:Cong} is equivalent to
equation $xm=yn$ over the field $\FF_p$.

Throughout, the notations $U=O(V)$ and $U\ll V$ are each equivalent to the
statement that the inequality $|U|\le c\,V$ holds with a
constant $c>0$ which may depend (where obvious) on the integer
$\ell\ge 1$ or the real number $\eps>0$. 

\subsection{Initial estimates}
\label{sec:initest}

It is straightforward to show that the estimate
\begin{equation}
\label{eq:Basic}
J(H,\cM)=\frac{H^2M^2}{p} + O(H M^{3/2}\,\log p)
\end{equation}
holds. Indeed,~\eqref{eq:Basic} follows easily from~\eqref{eq:N char sum} 
below combined with a bound on power moments of character sums;
see, for example,~\cite[Theorem~2]{ACZ}. Moreover,
the same method shows that~\eqref{eq:Basic} holds for congruences
in which the variables $x,y$ lie in intervals of the form $[A+1, A+H]$. 

The stronger \textit{conditional} estimate
\begin{equation}
\label{eq:Moderate}
J(H, \cM) = \frac{H^2M^2}{p} + O\(H M p^{o(1)}\)
\qquad(p\to\infty)
\end{equation}
can be established using~\eqref{eq:N char sum}
in conjunction with a ``square-root cancellation'' bound
on character sums implied by the Generalised Riemann Hypothesis;
see, for example, Munsch and Shparlinski~\cite[Equation~(4)]{MuSh}.

Furthermore, it is shown in the proof
of a result of Garaev (see the bound on $J$ in the proof of~\cite[Lemma~2.1]{Gar})
that for any fixed integer $\ell\ge 1$ one has the upper bound
\begin{equation}
\label{eq:Gar-J}
J(H,\cM)\le\(Hp^{-1/\ell}+1\)H^{1+o(1)}M^{1+1/\ell}
\qquad(H\to\infty).
\end{equation}

In the present paper, we obtain several results which 
improve~\eqref{eq:Basic} and~\eqref{eq:Gar-J}.
Moreover, in a wide range
of the parameters $p,H,M$ we achieve an \textit{unconditional}
upper bound on $J(H,\cM)$ that has roughly
the same shape as~\eqref{eq:Moderate}.

As an application of this bound, we give a 
new estimate on certain trilinear sums of multiplicative characters. 
We also combine it with recent results of  Kowalski,  Michel and  Sawin~\cite{KMS2}
to derive a new bound on bilinear sums with Kloosterman sums
that extends the range of the applicability of \cite{KMS2}. 

\subsection{Notational conventions}

Throughout the paper,
the letter $p$ always denotes a prime number. 

We use $|\cA|$ to denote the cardinality of a finite  set $\cA$.  

The notations $U = O(V)$, 
$U \ll V$ and $ V\gg U$ are all used to indicate
that $|U|\leqslant c|V| $ hold for some absolute
constant $c>0$, and we write $U \asymp V$ if $U \ll V \ll  U$. 
We also use $U^{o(1)}$ to denote any function $f(U)$ such that
$\log |f(U)|/\log U \to 0$ as $U \to \infty$.

\section{Main results}
\label{sec:results}

\subsection{Congruences}

We use ideas of Heath-Brown~\cite{H-B2} to establish 
the following estimate. 

\begin{theorem}
\label{thm:cong} 
For an integer $H\in[1,p)$ and a set $\cM\subseteq \FF_p^*$
of cardinality $M$, the following holds as $p\to\infty$:
$$
J(H,\cM)\ll\begin{cases}
H^2M^2 p^{-1} + H Mp^{o(1)}
&~~~~\hbox{if $H\ge p^{2/3}$},\\
H^2M^2p^{-1}+HM^{7/4}p^{-1/4+o(1)}+M^2
&~~~~\hbox{if $H<p^{2/3}$ and $M\ge p^{1/3}$},\\
HMp^{o(1)}+M^2
&~~~~\hbox{if $H<p^{2/3}$ and $M<p^{1/3}$}.
\end{cases}
$$
\end{theorem}

For $M\le p^{1/3+o(1)}$, Theorem~\ref{thm:cong} yields the following 
unconditional variant of the conditional estimate~\eqref{eq:Moderate}:
\begin{equation}
\label{eq:large H}
J(H,\cM) \ll H^2M^2 p^{-1} + H Mp^{o(1)}+M^2.
\end{equation}
In particular, we have~\eqref{eq:Moderate} if 
$$
H \ge p^{2/3+o(1)} \qquad \text{or} \qquad M\le  \min\{H, p^{1/3}\} p^{o(1)}.
$$ 

\subsection{Trilinear character sums}
\label{sec:char sum}

We use $\cX$ to denote the set of
multiplicative characters of $\FF_p^*$,
and $\cX^*=\cX\setminus\{\chi_0\}$ is the set of   
\textit{nonprincipal} characters; for the relevant background
on multiplicative characters, we refer the reader to
Iwaniec and Kowalski~\cite[Chapter~3]{IwKow}.

We now give some applications to bounds of certain
trilinear character sums.
Specifically, for an integer $H\in[1,p)$, two sets
$\cK,\cM\subseteq \FF_p^*$,
a character $\chi\in\cX^*$, and arbitrary complex weights
$\bfalpha=(\alpha_h)_{h=1}^H$, $\bfzeta=(\zeta_k)_{k\in\cK}$
and $\bfeta=(\eta_m)_{m\in\cM}$, we define
$$
W_\chi(H,\cK,\cM;\bfalpha,\bfzeta,\bfeta)
=\sum_{h=1}^H\sum_{k\in\cK}
\sum_{m\in\cM}\alpha_h\zeta_k\eta_m\,\chi(h+km). 
$$

Similar sums with only one set in $\FF_p^*$ have been previously studied;
a variety of bounds and corresponding applications
can be found in~\cite{BKS2,Chang,Kar,ShkShp}.
To simplify the formulation the next theorem
(and since it is the most interesting case in view of, e.g.,
the results of Chang~\cite{Chang}) we consider only the case in which
one of the sets has its cardinality bounded above by $p^{1/3+o(1)}$
(thus, the bound~\eqref{eq:large H} is at our disposal).

\begin{theorem}
\label{thm:char-3} 
Let the notation be as above, and let $K$ and $M$
denote the cardinalities of $\cK$ and $\cM$, respectively.
Suppose that $M\le p^{1/3+o(1)}$
and that the weights $\alpha_h,\zeta_k,\eta_m$
are all bounded by one in absolute value. Then,
for any fixed integer $\ell\ge 1$ we have 
\begin{align*}
&|W_\chi(H,\cK,\cM;\bfalpha,\bfzeta,\bfeta)|\\
&\qquad\le HKM\(p^{-1/2\ell}+H^{-1/2\ell}M^{-1/2\ell}+H^{-1/\ell}\)
\(p^{1/4\ell}+K^{-1/2}p^{1/2\ell}\)p^{o(1)}.
\end{align*}
\end{theorem}

In particular, if under the conditions of Theorem~\ref{thm:char-3}
we have $K\ge p^{\eps}$ for some fixed $\eps>0$,
then taking $\ell$ large enough to guarantee that 
$p^{1/4\ell}\ge K^{-1/2}p^{1/2\ell}$, the bound of Theorem~\ref{thm:char-3} 
becomes
\begin{align*}
&|W_\chi(H,\cK,\cM;\bfalpha,\bfzeta,\bfeta)|\\
&\qquad\qquad\le HKM\(p^{-1/4\ell}
+(HMp^{-1/2})^{-1/2\ell}
+(Hp^{-1/4})^{-1/\ell}\)p^{o(1)}.
\end{align*}
Thus, we obtain the following corollary.

\begin{corollary}
\label{cor:char} For any $\eps> 0$ there exists $\kappa > 0$
with the following property. Let the notation be as in Theorem~\ref{thm:char-3}.
Suppose that
$$
K\ge p^\eps,\qquad
M\le p^{1/3+o(1)},\qquad
HM\ge p^{1/2+\eps}\mand
H\ge p^{1/4+\eps}.
$$
Suppose further that the weights $\alpha_h,\zeta_k,\eta_m$
are all bounded by one in absolute value. Then
$$
W_\chi(H,\cK,\cM;\bfalpha,\bfzeta,\bfeta)\ll HKM p^{-\kappa}.
$$
\end{corollary}

It is clear that Corollary~\ref{cor:char}, coupled with standard techniques using
bivariate shifts $h\mapsto h+uv$, can be used to obtain nontrivial
estimates for double sums
$$
S_{\chi}(H,\cM)=\sum_{m\in\cM}\left|\sum_{h=1}^H\chi(h+m)\right|.
$$

\subsection{Bilinear sums of Kloosterman sum}
\label{sec:Kloos sum}

Next, we consider multidimensional Kloosterman sums of the form
$$
K_r(n)=\frac{1}{p^{(r-1)/2}} \sum_{\substack{x_1, \ldots, x_r =1\\ 
x_1 \cdots x_r \equiv n \bmod  p}}^{p-1} \ep(x_1 + \cdots + x_r),
$$
where $\e(t)=e^{2\pi i t/p}$ for all $t\in\RR$.
By the classical Deligne bound we have
\begin{equation}
\label{eq:Deligne}
|K_r(n)| \le r;
\end{equation}
see, for example,~\cite[Equation~(11.58)]{IwKow} and the follow-up discussion.

Recently, motivated by an abundance of applications,
there has been considerable interest in the estimation of
weighted sums of Kloosterman sums
$$
\SAMN=\sum_{m\in\cM}\sum_{n\in\cN}\alpha_m K_r(mn)
$$
with two sets $\cM, \cN \subseteq  \FF_p^*$ and complex weights
$\bfeta=\(\eta_m\)_{m \in \cM}$, and also
$$
\SABMN=\sum_{m\in\cM}\sum_{n\in\cN}\alpha_m\beta_n K_r(mn)
$$
with complex weights $\bfalpha=\(\alpha_m\)_{m\in\cM}$
and $\bfbeta=\(\beta_n\)_{n\in\cN}$;  
see~\cite{BFKMM1,BFKMM2,FKM,KMS1,KMS2,Shkr,Shp,ShpZha}
and the references therein. Of course, only bounds that are superior
to those that follow directly from~\eqref{eq:Deligne}
are of interest and use. 

The bound  of Theorem~\ref{thm:cong} can be embedded
in the arguments of~\cite{BFKMM1,KMS1,KMS2} which,
for both sums $\SAMN$ and $\SABMN$, rely on a bound for $J(H,\cM)$
in the special that $\cM=\{1,\ldots,M\}$ for some positive integer $M$.
In the present paper, we demonstrate the idea only
in the simple case of the sums  $\SAMN$.

\begin{theorem}
\label{thm:Kloost}
For any set $\cM\subseteq \FF_p^*$ of cardinality $M$, an interval
$\cN\subseteq \FF_p^*$ of length $N<p$, complex weights $\alpha_m$ bounded by one
in absolute value, and a fixed even integer $\ell\ge 1$, we have
\begin{align*}
|\SAMN|&\le MN\(N^{-1/2\ell}
+M^{-1/8\ell}N^{-1/\ell}p^{3/8\ell+1/2\ell^2}\right.\\
&\qquad\qquad\left.+M^{-1/2\ell}N^{-1/\ell}p^{1/2\ell+1/2\ell^2}
+N^{-3/2\ell}p^{1/2\ell+1/\ell^2}\)p^{o(1)}.
\end{align*}
\end{theorem}

\begin{remark}
The results of~\cite{Shp} only apply to sums $\SAMN$ with $r=2$
(i.e., to classical one dimensional Kloosterman sums); they hold
arbitrary sets $\cM\subseteq \FF_p^*$ and intervals
$\cN\subseteq \FF_p^*$ of length $N<p$.
Moreover, in the special case that $\cM=\{1,\ldots,M\}$,~\cite[Theorem~2.1]{Shp} gives the strongest known bound
in the crucial range where $M$ and $N$ are both of size $p^{1/2+o(1)}$.
We remark, however, that if $M$ and $N$ are comparable in size,
then the bound of~\cite[Theorem~2.1]{Shp} is nontrivial
only if both quantities exceed $p^{4/9+\varepsilon}$ for
some fixed $\eps>0$, whereas Theorem~\ref{thm:Kloost} is nontrivial
once both quantities exceed $p^{1/3+\varepsilon}$.
\end{remark}

\begin{remark}
Assuming that $|\alpha_m|\le 1$ for all $m\in\cM$, the bounds 
of~\cite[Theorem~4.2]{KMS2}
imply that for any fixed integer
$\ell\ge 1$, the bound
\begin{equation}
\label{eq:KMS Bound}
|\SAMN|\le MN\(M^{-1/2\ell}N^{-1/\ell}p^{1/2\ell+1/2\ell^2}\)p^{o(1)}
\end{equation}
holds provided that the integers $N$ and the set $\cM\subseteq \FF_p^*$
satisfy at least one of the two conditions
\begin{equation}
\label{eq:KMS Cond 1}
p^{1/\ell} \le N \le \tfrac12 p^{1/2 + 1/ 2\ell}
\end{equation}
or 
\begin{equation}
\label{eq:KMS Cond 2}
p^{1/\ell}\le N\mand NM^+\le \tfrac12 p^{1+1/2\ell},
\end{equation}
where (recalling the convention~\eqref{eq:Fp repr}) we denote
$$
M^+=\max_{m \in \cM} m.
$$
In fact, the proof of~\cite[Theorem~4.2]{KMS2} can be easily
extended to cover arbitrary intervals $\cN=\{s+1,\ldots, s+N\}$ 
(not only initial intervals of the form
$\{1,\ldots,N\}$) and arbitrary sets
$\cM\subseteq \FF_p^*$.  Indeed, in the case~\eqref{eq:KMS Cond 1}
it is enough to use a result of Ayyad, Cochrane and
Zheng~\cite[Theorem~1]{ACZ} in the appropriate place
(where the congruence $a_1n_2 \equiv a_2 n_1 \bmod p$ is replaced with 
an equation $a_1n_2 =a_2 n_1$ over $\ZZ$), whereas in
the case~\eqref{eq:KMS Cond 2} no changes are required. 

Furthermore, under~\eqref{eq:KMS Cond 1} the argument in the proof 
of~\cite[Theorem~4.2]{KMS2} applies to arbitrary sets $\cM$, but in the case 
of  the condition~\eqref{eq:KMS Cond 2}  the restriction on the size of 
$M^+$ is crucial. 

Our Theorem~\ref{thm:Kloost} gives a somewhat weaker bound
than~\eqref{eq:KMS Bound}, but it applies in greater generality without 
any restriction on $M^+$.

To illustrate this, let $\cM\subseteq \FF_p^*$ be such that
$M^+>p^{1/2}$; then both~\eqref{eq:KMS Cond 1}
and~\eqref{eq:KMS Cond 2} require that
$N\le\tfrac12 p^{1/2+1/2\ell}$. If $N=p^{15/26+o(1)}$ (say),
then the bound~\eqref{eq:KMS Bound} only applies if $\ell\le 6$,
so it yields a nontrivial bound only if
$$
MN^2\ge p^{1+1/\ell+o(1)}\ge p^{7/6+o(1)}\qquad(\ell\le 6);
$$
this requires that $M\ge p^{1/78+o(1)}$.
On the other hand, Theorem~\ref{thm:Kloost},
for an appropriate choice of $\ell$ (which is
not restricted by any conditions similar to~\eqref{eq:KMS Cond 1} and~\eqref{eq:KMS Cond 2})
provides a nontrivial bound already for $M\ge p^{\varepsilon}$.
\end{remark}

\begin{remark} As in~\cite{KMS1,KMS2}, we expect that Theorem~\ref{thm:Kloost}
can be extended to a broad class of trace functions satisfying suitable
``big monodromy'' assumptions and other natural hypotheses.
\end{remark}

\section{Preparations} 

\subsection{Character sums}

For any $z\in\FF_p$ we have the
well known orthogonality relation (see~\cite[Section~3.1]{IwKow}):
\begin{equation}
\label{eq:Orth}
\sum_{\chi\in\cX}\chi(z)=\begin{cases}
p-1&\quad\text{if $z=1$},\\
0&\quad\text{otherwise}.
\end{cases}
\end{equation}
Using~\eqref{eq:Orth} we derive that
\begin{equation}
\label{eq:N char sum}
\begin{split}
J(H,\cM)&=
\sum_{x,y=1}^H \sum_{m,n\in\cM}
\frac{1}{p-1}\sum_{\chi\in\cX}\chi(xy^{-1}mn^{-1})\\
&=\frac{1}{p-1}\sum_{\chi\in\cX}
\left|\sum_{x=1}^H\chi(x)\right|^2~
\left|\sum_{m\in\cM}\chi(m)\right|^2,
\end{split}
\end{equation}
which is used in our proof of Theorem~\ref{thm:cong}.

We also need the Weil bound on multiplicative character sums
in the following form; see~\cite[Theorem~11.23]{IwKow}.

\begin{lemma}
\label{lem:Weil}
For square-free and coprime polynomials $f,g\in\FF_p[X]$  we have
$$
\max_{\chi\in\cX^*}\bigg|\sum_{a\in\FF_p}
\chi(f(a))\cdot\overline\chi(g(a))\bigg|\ll\(\deg f+\deg g\)p^{1/2}.
$$
\end{lemma}

\subsection{Recursive inequality}

In the proof of Theorem~\ref{thm:cong}, the following lemma is used to bound
$J(H,\cM)$ in terms of $J(K,\cM)$ with a suitably chosen $K$.

\begin{lemma}
\label{lem:HtoK}
For any integers $H,K\in[1,p)$ we have the 
uniform estimate
\begin{equation}
\label{eq:JHMJKM}
KJ(H,\cM)=HJ(K,\cM)+O\((H+K)(HKp^{-1}+1)M^2\).
\end{equation}
\end{lemma}

\begin{proof}
By symmetry, we can assume that $K\le H$.

For a fixed pair $(m,n)\in\cM^2$, let $\Lambda(m,n)$ be the lattice
defined by
$$
\Lambda(m,n) = \{(x,y) \in\ZZ^2: xm\equiv yn\bmod p\}, 
$$
which clearly has determinant
$\det\Lambda(m,n)=p$. 
By~\cite[Lemma~1]{H-B1} there are
basis vectors $\vec{e},\vec{f}\in\Lambda(m,n)$ 
such that the Euclidean lengths $\| \vec{e}\|$ and $\| \vec{f}\|$ satisfy
\begin{equation}
\label{eq:small basis}
\| \vec{e}\| \le \| \vec{f}\| \mand p \ll \| \vec{e}\|   \| \vec{f}\| \ll p
\end{equation}
and such that for any vector $\vec{u}=a\,\vec{e}+b\,\vec{f}$ with $a,b\in\ZZ$
one has
\begin{equation}
\label{eq:small coeffs}
|a| \le \frac{c_0\|\vec{u}\|}{\|\vec{e}\|}\mand |b|\le \frac{c_0\|\vec{u}\|}{\|\vec{f}\|}  
\end{equation}
with some absolute constant $c_0>0$. 

First, we consider pairs $(m,n)$ in $\cM^2$ for which there exist
$\vec{e},\vec{f}\in\Lambda(m,n)$ as above
with $\|\vec{f}\|\le 2^{1/2}c_0 H$.
Using~\eqref{eq:small basis} and~\eqref{eq:small coeffs},
the total number of
vectors $\vec{u}=(x,y)\in\Lambda(m,n)$ with $x,y\in[1,H]$ is at most 
$$
 \(\frac{2^{3/2}c_0H}{\|\vec{e}\|}+1\)
\(\frac{2^{3/2}c_0H}{\|\vec{f}\|}+1\)
\ll\frac{H^2}{\|\vec{e}\|\|\vec{f}\|}\ll \frac{H^2}{p}.
$$
It follows that the contribution to $J(H,\cM)$
from all such pairs is $O(H^2M^2 p^{-1})$.
Similarly, the number of
vectors $\vec{u}=(x,y)\in\Lambda(m,n)$ with $x,y\in[1,K]$ is at most 
$$
 \(\frac{2^{3/2}c_0K}{\|\vec{e}\|}+1\)
\(\frac{2^{3/2}c_0K}{\|\vec{f}\|}+1\)
\ll\frac{K^2}{\|\vec{e}\|\|\vec{f}\|}
+\frac{K}{\|\vec{e}\|}+\frac{K}{\|\vec{f}\|}+1\ll \frac{HK}{p}+1
$$
since
$$
\frac{K}{\|\vec{f}\|}\le\frac{K}{\|\vec{e}\|}\mand
\frac{K}{\|\vec{e}\|}\ll\frac{K\|\vec{f}\|}{p}\ll\frac{HK}{p}.
$$
Hence, the contribution to $J(K,\cM)$ from all such pairs
is $O(HKM^2 p^{-1}+M^2)$.

Let $\sM$ be the set formed from the remaining pairs $(m,n)$ in $\cM^2$.
For each pair in $\sM$, fix a choice of basis vectors $\vec{e},\vec{f}\in\Lambda(m,n)$
satisfying~\eqref{eq:small basis} and~\eqref{eq:small coeffs}.
Taking into account that $\|\vec{f}\|>2^{1/2}c_0H$, if
$$
\vec{u}=(x,y)=a\,\vec{e}+b\,\vec{f}\in\Lambda(m,n)
$$
for some $x,y\in[1,H]$ and $a,b\in\ZZ$,
then $b=0$ by~\eqref{eq:small coeffs}; in other words,
\begin{equation}
\label{eq:uae}
\vec{u} = a\,\vec{e} \in [1,H]^2. 
\end{equation}
Writing $\e=(r,s)$, this implies that the numbers $r,s$ have
the same sign, and $rs\ne 0$. Replacing $\vec{e}$ by $-\vec{e}$ we can assume
 $r,s\ge 1$. Now one sees easily that the number of vectors $\vec{u}$ satisfying
\eqref{eq:uae} is precisely $\fl{H/\|\vec{e}\|_\infty}$,
where $\|\vec{e}\|_\infty=\max\{r,s\}$.

We now use  $\vec{e}_{m,n}$ to denote the corresponding vector $\vec{e}$ coming from a given 
pair $(m,n)\in\sM$. 
Thus, the contribution to $J(H,\cM)$ from the pairs in $\sM$ is
$$
\sum_{(m,n)\in\sM} \fl{H/\|\vec{e}_{m,n}\|_\infty} = 
H\sum_{(m,n)\in\sM}\|\vec{e}_{m,n}\|_\infty^{-1}+O(M^2).
$$
Since $H\ge K$, a similar result holds with $H$ replaced by $K$.

Combining all of the above results, we deduce the following estimates:
\begin{align*}
KJ(H,\cM)&=HK\sum_{(m,n)\in\sM}\|\vec{e}_{m,n}\|_\infty^{-1}+O(H^2KM^2 p^{-1}+KM^2),\\
HJ(K,\cM)&=HK\sum_{(m,n)\in\sM}\|\vec{e}_{m,n}\|_\infty^{-1}+O(H^2KM^2 p^{-1}+HM^2),
\end{align*}
and these imply~\eqref{eq:JHMJKM} for $K\le H$.
\end{proof}

\section{Proofs of Main Results}

\subsection{Proof of Theorem~\ref{thm:cong}}

We first bound $J(K,\cM)$ with a suitably
chosen $K$, and then we apply Lemma~\ref{lem:HtoK}.

Replacing $H$ with $K$ in~\eqref{eq:N char sum}
and using the Cauchy inequality, we derive that
\begin{align*}
J(K,\cM)^2 
&\le\frac{1}{(p-1)^2}\sum_{\chi\in\cX}
\left|\sum_{m\in\cM}\chi(m)\right|^2
\sum_{\chi\in\cX}
\left|\sum_{x=1}^K\chi(x)\right|^4 
\left|\sum_{m\in\cM}\chi(m)\right|^2\\
&=\frac{M}{p-1} \sum_{\chi\in\cX}
\left|\sum_{x=1}^K\chi(x)\right|^4 
\left|\sum_{m\in\cM}\chi(m)\right|^2,
\end{align*}
where we have used~\eqref{eq:Orth} in the second step.
Since $\chi$ is multiplicative this implies
\begin{equation}
\label{eq:N2 K2}
J(K,\cM)^2 \ll \frac{M}{p} 
\sum_{\chi\in\cX} \left|\sum_{u=1}^L \rho(u)\chi(u)\right|^2~
\left|\sum_{m\in\cM}\chi(m)\right|^2, 
\end{equation}
where $L=K^2$, and $\rho(u)$ is the number of pairs $(x,y)\in[1,K]^2$
such that $u=xy$. 
Taking into account~\eqref{eq:Orth} once again, we see that 
\begin{equation}
\label{eq:K2 N}
\sum_{\chi\in\cX} \left|\sum_{u=1}^L \rho(u)\chi(u)\right|^2~
\left|\sum_{m\in\cM}\chi(m)\right|^2 = p N, 
\end{equation}
where
$$
N=\sum_{\substack{(u,v,m,n)\\um\equiv vn\bmod p}}\rho(u)\rho(v).
$$
Since $L\le p^2$, using a trivial bound on the divisor function
(see, for example,~\cite[Equation~(1.81)]{IwKow}) we have
$$
N\le J(L,\cM) p^{o(1)}\qquad (p\to\infty).
$$
Combining this bound with~\eqref{eq:N2 K2} and~\eqref{eq:K2 N}, 
it follows that
$$
J(K,\cM)  \ll J(L,\cM)^{1/2} M^{1/2}p^{o(1)},
$$
and so by Lemma~\ref{lem:HtoK} it follows that
\begin{equation}
\label{eq:NL}
J(H,\cM)\ll HK^{-1}M^{1/2}J(L,\cM)^{1/2}p^{o(1)}+(H/K+1)(HKp^{-1}+1)M^2.
\end{equation}

In the case that $H \ge p^{2/3}$, we take $K=\fl{H^{1/2}}$.
Since $L\le H$, we conclude from~\eqref{eq:NL} that
$$
J(H,\cM)\ll H^2M^2p^{-1}+H^{1/2}M^{1/2}J(H,\cM)^{1/2}p^{o(1)},
$$
which implies that
\begin{equation}
\label{eq:scantheroom}
J(H,\cM)\ll H^2M^2 p^{-1} + H Mp^{o(1)}\qquad(H \ge p^{2/3}).
\end{equation}
This proves the theorem in this case.

Next, suppose that $H<p^{2/3}$ and $M\ge p^{1/3}$.
Let $K=\rf{M^{1/4}p^{1/4}}$. Since $K\ge p^{1/3}$ and so
$L\ge p^{2/3}$, using~\eqref{eq:scantheroom} with $H=L$ we see that
$$
J(L,\cM)\ll K^4M^2p^{-1}+K^2Mp^{o(1)}\ll M^3p^{o(1)},
$$
hence by~\eqref{eq:NL} we have
\begin{align*}
J(H,\cM)&\ll HK^{-1}M^2p^{o(1)}+(H/K+1)(HKp^{-1}+1)M^2\\
&\ll H^2M^2p^{-1}+HM^{7/4}p^{-1/4+o(1)}
+HM^{9/4}p^{-3/4}+M^2.
\end{align*}
The third term is dominated by the second term since $M<p$,
thus it can be dropped, and the theorem is proved in this case.

Finally, suppose that $H<p^{2/3}$ and $M<p^{1/3}$.
Put $K=\rf{p^{1/3}}$. Since $L\ge p^{2/3}$, using~\eqref{eq:scantheroom}
with $H=L$ it follows that
$$
J(L,\cM)\ll Mp^{2/3+o(1)},
$$
hence by~\eqref{eq:NL}, our choice of $K$, and the fact
that $HM<p$, the theorem follows in this case. This concludes the proof.

\subsection{Proof of Theorem~\ref{thm:char-3}}
For each $\lambda\in\FF_p^*$ let
$$
\vartheta(\lambda)=\left|\{(h,m)\in[1,H]\times\cM:
hm^{-1}\equiv\lambda\bmod p\}\right|. 
$$
Then, using the multiplicativity of $\chi$, we have
\begin{align*}
\left|W_\chi(H,\cK,\cM;\bfalpha,\bfzeta,\bfeta)\right|
&\le\sum_{m\in\cM}\sum_{h=1}^H\left|\sum_{k\in\cK}
\zeta_k\chi(hm^{-1}+k)\right|\\
&\le\sum_{\lambda\in\FF_p^*}
\vartheta(\lambda)\left|\sum_{k\in\cK}\zeta_k\chi(\lambda+k)\right|.
\end{align*}
Clearly,
$$
\sum_{\lambda\in \FF_p^*} \vartheta(\lambda) = H M \mand
\sum_{\lambda\in \FF_p^*} \vartheta(\lambda)^2 = J(H,\cM).
$$
For any fixed integer $\ell\ge 1$ we write
$$
\vartheta(\lambda)=\vartheta(\lambda)^{(\ell-1)/\ell}\cdot
\(\vartheta(\lambda)^2\)^{1/2\ell}
$$
and applying the H{\"o}lder inequality, obtaining
\begin{align*}
&\left|W_\chi(H,\cK,\cM;\bfalpha,\bfzeta,\bfeta)\right|^{2\ell}\\
&\qquad\quad\le(HM)^{2\ell-2}J(H,\cM)
\sum_{\lambda\in\FF_p^*}\left|\sum_{k\in\cK}\zeta_k\chi(\lambda+k)\right|^{2\ell}\\
&\qquad\quad=(HM)^{2\ell-2}J(H,\cM)\sum_{j_1,k_1,\ldots,j_\ell,k_\ell\in\cK}
\prod_{i=1}^\ell\zeta_{j_i}\overline\zeta_{k_i}\cdot 
\sum_{\lambda\in\FF_p^*}\prod_{i=1}^\ell
\chi(\lambda+j_i)\overline\chi(\lambda+k_i)\\
&\qquad\quad\le(HM)^{2\ell-2}J(H,\cM)\sum_{j_1,k_1,\ldots,j_\ell,k_\ell\in\cK}
\left|\sum_{\lambda\in\FF_p^*}\prod_{i=1}^\ell
\chi(\lambda+j_i)\overline\chi(\lambda+k_i)\right|,
\end{align*}
where $\overline\chi$ is the conjugate character. 

We now apply Lemma~\ref{lem:Weil} to the sums over $\lambda$ 
when the sets 
$$\{j_1,   \ldots, j_\ell\} \mand \{k_1,   \ldots, k_\ell\}
$$ 
are different, and we use the trivial bound otherwise; this leads to
the bound 
$$
W_\chi(H,\cK,\cM;\bfalpha,\bfzeta,\bfeta)^{2\ell}
\ll(HM)^{2\ell-2}J(H,\cM)(K^{2\ell}p^{1/2}+K^\ell p).
$$
Since $M \le p^{1/3+o(1)}$,  we can apply~\eqref{eq:large H} to derive that 
\begin{align*}
&W_\chi(H,\cK,\cM;\bfalpha,\bfzeta,\bfeta)^{2\ell}\\
&\qquad\qquad\le (HKM)^{2\ell}(p^{-1}+H^{-1}M^{-1}+H^{-2})
(p^{1/2}+K^{-\ell}p)p^{o(1)}, 
\end{align*}
and the result follows. 

\subsection{Proof of Theorem~\ref{thm:Kloost}}

In what follows, we use notation of the form $x\sim X$ as an
abbreviation for $X<x\le 2X$.

To prove Theorem~\ref{thm:Kloost}, we follow the strategy
of the proof of~\cite[Theorem~4.2]{KMS2}, which is summarised in~\cite[\S4.2]{KMS2}.
Let $A$ and $B$ be integer parameters for which 
\begin{equation}
\label{eq:AB N}
2AB\le N.
\end{equation}
Then, as in~\cite[Section~4.2]{KMS2}
and~\cite[Chapter~IV]{FM98} we have
\begin{align*}
|\SAMN|\ll\frac{\log p}{AB}
\sum_{u\in\FF_p}\sum_{v\in\FF_p^*}\nu(u,v)
\left|\sum_{b\sim B}\eta_bK_r(u(v+b))\right|
\end{align*}
with some complex weights $\bfeta=(\eta_b)_{b\sim B}$
of absolute value one, and
\begin{align*}
\nu(u,v)&=\sssum_{\substack{a\sim A,\,m\in\cM,\,n\in[1,N]\\
am\equiv u\bmod p,\;av\equiv n\bmod p}}|\alpha_m|\\
&\le\big|\{(a,m,n)\in\cA\times\cM\times\cN:am\equiv u\bmod p, 
\ av\equiv n\bmod p\}\big|, 
\end{align*}
where $\cA$ denotes the set of integers $a\sim A$.

Following~\cite[Section~4.2]{KMS2} (see also the proof of Theorem~\ref{thm:char-3}), we use the H{\"o}lder inequality
to derive that
\begin{equation}
\label{eq:S nu K}
|\SAMN|^{2\ell}\le\frac{S_1^{2\ell-2}S_2S_K}{(AB)^{2\ell}}p^{o(1)}
\end{equation}
for any fixed integer $\ell\ge 1$, where 
$$
R_1=\sum_{u\in\FF_p}\sum_{v\in\FF_p^*}\nu(u,v) \mand 
R_2=\sum_{u\in\FF_p}\sum_{v\in\FF_p^*}\nu(u,v)^2
$$
and 
$$
S_K=\sum_{u\in\FF_p}\sum_{v\in\FF_p^*}
\left|\sum_{b\sim B}\eta_bK_r(u(v+b))\right|^{2\ell}.
$$ 
As in~\cite{KMS2} we have trivially
\begin{equation}
\label{eq:nu1}
R_1\ll AMN,
\end{equation}
and also
\begin{align*}
R_2\ll
&\big|\{(a_1,a_2,m_1,m_2,n_1,n_2)
\in\cA^2\times\cM^2\times\cN^2:\\
&\qquad\qquad
a_1m_1\equiv a_2m_2\bmod p,~a_1n_1\equiv a_2n_2\bmod p\}\big|. 
\end{align*}
There are at most $J(2A,\cM)$ solutions
$(a_1,a_2,m_1,m_2)\in\cA^2\times\cM^2$ to the congruence
$a_1m_1 \equiv a_2m_2 \bmod p$,
and for any such solution, there are at most $N$ pairs
$(n_1,n_2)\in\cN^2$ such that
$a_1n_1\equiv a_2n_2\bmod p$; therefore,
\begin{equation}
\label{eq:nu2}
R_2\ll N J(2A,\cM).
\end{equation}
Substituting~\eqref{eq:nu1} and~\eqref{eq:nu2}
in~\eqref{eq:S nu K}, we obtain that
\begin{equation}
\label{eq:S Sigma}
|\SAMN|^{2\ell}
\le A^{-2}B^{-2\ell}M^{2\ell-2}N^{2\ell-1} J(2A,\cM)S_Kp^{o(1)}.
\end{equation}

We turn our attention to the sum $S_K$. Estimating $S_K$
lies at the heart of the method of~\cite[Section~4.2]{KMS2},
where the bound
\begin{equation}
\label{eq:Sigma gamma}
S_K\ll B^\ell p^2 + B^{(2-\gamma) \ell} p^{3/2} + B^{2\ell} p
\end{equation}
with some fixed $\gamma\ge 0$
is derived from~\cite[Lemma~2.3]{KMS2} and~\cite[Theorem~4.4]{KMS2}.
Moreover, from the statement and proof of~\cite[Theorem~4.3]{KMS2}
we see that the way $\gamma$ is defined, 
the product 
$\gamma\ell$ is an integer not less than $(\ell-1)/2$; for even $\ell$
this implies that $\gamma\ge 1/2$, and so
$$
B^{(2-\gamma)\ell}p^{3/2}\le B^{3\ell/2}p^{3/2}
\le\max\{B^\ell p^2,B^{2\ell}p\}.
$$
Thus, \eqref{eq:Sigma gamma} simplifies to 
\begin{equation}
\label{eq:Sigma}
S_K\ll B^\ell p^2 + B^{2\ell} p. 
\end{equation}

Taking $B = \fl{p^{1/\ell}}$, the bound~\eqref{eq:Sigma} 
becomes $S_K\ll B^{2\ell}p$; using this bound in
\eqref{eq:S Sigma} along with the bound for $J(2A,\cM)$
afforded by Theorem~\ref{thm:cong}, we get that
$$
|\SAMN|^{2\ell}
\le A^{-2}M^{2\ell-2}N^{2\ell-1}
(A^2M^2p^{-1}+AM^{7/4}p^{-1/4}+AM+M^2)p^{1+o(1)}.
$$
We now choose 
$$
A=\fl{\frac{N}{2B}}\asymp N p^{-1/\ell},
$$
which guarantees that the condition~\eqref{eq:AB N} is met,
and after simple calculations we obtain the stated bound. 

\section{Comments}

Iwaniec and   S{\'a}rk{\"o}zy~\cite{IwSar}  have considered a question about the distance
between the product set of two sufficiently dense sets of integers in the interval $[1,N]$
and the set of perfect squares. The same question can also be considered modulo $p$, 
which immediately leads to the question of obtaining nontrivial bounds on the trilinear  sums
$W_{\chi}(H,\cK, \cM;\bfalpha,\bfzeta,\bfeta)$ as in Section~\ref{sec:char sum} with a quadratic character $\chi$.

\section*{Acknowledgements} 

The authors are grateful to Roger Heath-Brown 
for several very useful discussions and
for making available a preliminary version of~\cite{H-B2}. 

This work was supported  in part by  
the  Australian Research Council  Grant DP170100786  (for I.~E.~Shparlinski).

\end{document}